\documentclass[a4paper,11pt]{amsart}
\usepackage{amssymb}
\usepackage{latexsym}
\usepackage{amsmath}
\usepackage{enumerate}
\usepackage{amsmath, hyperref}

\usepackage{tikz}
\usepackage{geometry}

\newtheorem{theorem}{Theorem}[section]
\newtheorem{lemma}[theorem]{Lemma}
\newtheorem{corollary}[theorem]{Corollary}
\newtheorem{proposition}[theorem]{Proposition}

\theoremstyle{definition}
\newtheorem{definition}[theorem]{Definition}
\newtheorem{question}[theorem]{Question}
\newtheorem{fact}[theorem]{Fact}

\newtheorem*{remark}{Remark}
\newtheorem*{claim}{Claim}
\newtheorem*{subclaim}{Subclaim}

\newcommand{\lhdneq}{\mathrel{\lhd_{\hspace*{-0.235cm}{}_{\not -}}}}

\numberwithin{equation}{section}

\title{The limitless First Incompleteness Theorem}

\author{Yong Cheng\\
School of Philosophy, Wuhan University, China}
\address{School of Philosophy, Wuhan University, Wuhan 430072 Hubei, Peoples Republic of China}

\email{world-cyr@hotmail.com}

\thanks{We thank all referees for helpful suggestions and comments for improvements. We thank Albert Visser for communications of his work from which we know the current status of this field and some results in this paper are inspired. We also thank Harvey Friedman, Fedor Pakhomov and Emil Je\v{r}\'{a}bek for communications of their work.}



\begin{document}
\begin{abstract}
This work is motivated by the problem of finding the limit of the applicability of the first incompleteness theorem ($\sf G1$). A natural question is: can we find a minimal theory  for which $\sf G1$ holds? We examine the Turing degree structure of recursively enumerable (RE) theories for which $\sf G1$ holds and the interpretation degree structure of RE theories weaker than the theory $\mathbf{R}$ with respect to interpretation for which $\sf G1$ holds. We answer all questions that we posed in \cite{Cheng20}, and prove more results about them. 
It is known that there are no minimal essentially undecidable theories with respect to interpretation. We generalize this result and give some general characterizations which tell us under what conditions there are no minimal RE theories having some property with respect to interpretation. 
\end{abstract}

\thispagestyle{empty}
\newpage

\maketitle

\section{Introduction}

This work is motivated by the problem of finding the limit of the applicability of the first incompleteness theorem ($\sf G1$). 
In this paper, we work with  first-order theories with finite signature, and we equate a theory $T$ with the set of theorems provable in $T$. Let $T$ be a consistent recursively enumerable (RE) theory. To generalize $\sf G1$ to weaker theories than $\mathbf{PA}$ with respect to interpretation,  we introduce the notion ``$\sf G1$ holds for a theory $T$".

\begin{definition}[\cite{Cheng20}]\label{}
We say that \emph{$\sf G1$ holds for an  RE theory $T$}  if  any  consistent RE theory that interprets $T$  is incomplete.
\end{definition}

We say a theory $T$ is \emph{essentially incomplete} if  any consistent RE extension of $T$ in the same language is incomplete. 
The first incompleteness theorem tells us that $\mathbf{PA}$ is essentially incomplete. Motivated by finding the limit of the applicability of the first incompleteness theorem, we are interested in to what extent $\sf G1$ holds for theories weaker than $\mathbf{PA}$. The notion of essential incompleteness only refers to consistent RE extensions in the same language. To give a general formulation of $\sf G1$, in \cite{Cheng20} we use  the notion of interpretation such that we can compare different theories from distinctive languages. We could define the notion ``$\sf G1$ holds for a theory $T$" in different ways. For example, we could define this notion in a way that it is equivalent with $T$ is recursively inseparable or $T$ is effectively inseparable (for the definitions of these notions, we refer to Definition \ref{Inseparability}). This is a follow up paper of \cite{Cheng20}. In this paper, we just follow the definition of ``$\sf G1$ holds for a theory" in \cite{Cheng20}.

The theory of completeness/incompleteness  is closely related to the theory of decidability/undecidability (see \cite{undecidable}).  We say a theory $T$ is \emph{essentially undecidable} ({\sf EU}) if any consistent RE extension of $T$ in the same language is undecidable.

\begin{proposition}[\cite{undecidable}, see also \cite{Cheng20}]
For any consistent RE theory $T$, $\sf G1$ holds for  $T$ iff $T$ is essentially incomplete iff $T$ is essentially undecidable.
\end{proposition}
It is well known that $\sf G1$ holds for Robinson Arithmetic $\mathbf{Q}$ and the theory $\mathbf{R}$ (see \cite{undecidable}).

In this paper, the notion of interpretation we use closely follows \cite{Visser16, Visser24}. 
An \emph{interpretation} of a theory $U$ in a theory $V$ is based on a translation $\tau$ of the
$U$-language into the $V$-language. We will use the most inclusive notion to wit piecewise, multidimensional (with dimensions varying over pieces), relative, non-identity-preserving interpretability with parameters (\cite{Visser16}). A translation for the relational case commutes with the propositional connectives, and in some broad sense, it also commutes with the quantifiers but here there are a number
of extra features.  The translations could be more-dimensional in which a variable could be translated to an
appropriate sequence of variables; we may have domain relativisation allowing the range of the translated quantifiers
to be some domain definable in the $V$-language; the new domain may be built up from pieces of possibly different
dimensions; we also allow parameters in an interpretation, and the translation may specify a parameter-domain; moreover, identity need not be translated to identity but can be translated
to a congruence relation (\cite{Visser24}). 
We refer the reader for these definitions to \cite{Visser16, Visser24}.

We say a theory $T$ is \emph{interpretable} in a theory $S$ if there exists an
interpretation  of $T$ in $S$. Given theories $S$ and $T$, let $S\lhd T$ denote that $S$ is interpretable in $T$ (or $T$ interprets $S$); let $S\lhdneq T$ denote that  $T$ interprets $S$ but $S$ does not interpret  $T$; we say $S$ and $T$ are
\emph{mutually interpretable}, denoted by $S \bowtie T$, if $S\lhd T$ and $T\lhd S$. 
In fact, $\sf G1$ holds for many theories weaker than $\mathbf{PA}$ w.r.t. interpretation.

We have the following picture (for definitions of these weak theories, we refer to \cite{Cheng20}):\footnote{This is only a selective picture to give readers a sense that there are many theories weaker than $\mathbf{PA}$ for which ${\sf G1}$ holds. For the full picture of the hierarchy of weak arithmetic theories, we refer to \cite{Pudlak 93}.}
\begin{itemize}
\item $\mathbf{Q}\lhdneq  I\Sigma_1\lhdneq I\Sigma_2\lhdneq\cdots\lhdneq I\Sigma_n\lhdneq\cdots\lhdneq \mathbf{PA}$, and  ${\sf G1}$ holds for them.
    \item The theories $\mathbf{Q},  I\Sigma_0, I\Sigma_0+\Omega_{1}, \cdots, I\Sigma_0+\Omega_{n}, \cdots, B\Sigma_1, B\Sigma_1+\Omega_{1}, \cdots$, $B\Sigma_1+\Omega_{n}, \cdots$ are all mutually interpretable, and  ${\sf G1}$ holds for them.
  \item Theories $\mathbf{PA}^{-}, \mathbf{Q}^{+}, \mathbf{Q}^{-},  \mathbf{TC}, \mathbf{AS},\mathbf{S^1_2}$ and $\mathbf{Q}$  are all mutually interpretable, and ${\sf G1}$ holds for them.
\item  $\mathbf{R}\lhdneq\mathbf{Q}\lhdneq \mathbf{EA}\lhdneq \mathbf{PRA}\lhdneq\mathbf{PA}$, and ${\sf G1}$ holds for them.
\end{itemize}

A natural question is: can we find a minimal theory for which $\sf G1$  holds?  The answer of this question depends on our definition of minimality. Recall that we equate a
theory  with the set of theorems provable in it.
If we define minimality as having the minimal number of axioms, then any finitely axiomatizable essentially undecidable theory (e.g., Robinson Arithmetic $\mathbf{Q}$) is a minimal RE theory for which $\sf G1$  holds.

When we talk about minimality, we should specify the degree structure involved. Given two theories $U$ and $V$,   let \emph{$U\leq_{\sf T} V$}  denote that $U$  is Turing reducible to  $V$, and let \emph{$U<_{\sf T} V$}  denote that $U\leq_{\sf T} V$ but $V\nleq_{\sf T} U$. In \cite{Cheng20}, we examine two degree structures  that are respectively induced from Turing reducibility and interpretation: $\langle {\sf D_T}, \leq_{\sf T}\rangle$ and $\langle {\sf D_I}, \lhd\rangle$.
We first introduce two key notions. 

\begin{definition}[\cite{Cheng20}]~\label{}
\begin{enumerate}[(1)]
  \item Let ${\sf D_T}=\{S: S$ is RE, $S<_{\sf T} \mathbf{R}$ and $\sf G1$  holds for $S$\}.
  \item Let ${\sf D_I}=\{S: S$ is RE, $S\lhdneq \mathbf{R}$ and $\sf G1$  holds for $S$\}.
\end{enumerate}
\end{definition}

In \cite{Cheng20}, we show that there are no  minimal RE theories w.r.t. Turing reducibility for which $\sf G1$  holds, and prove some results about the structure  $\langle {\sf D_I}, \lhd\rangle$.
It is a question in \cite{Cheng20} that whether $\langle {\sf D_I}, \lhd\rangle$ has a minimal element?   This question is answered negatively in \cite{PV} (c.f., Theorem \ref{PV thm}).
The following questions about $\langle {\sf D_I}, \lhd\rangle$  are unanswered in \cite{Cheng20}:

\begin{question}~\label{key qn}
\begin{enumerate}[(1)]
  \item Can we show that for any Turing degree $\mathbf{0}<\mathbf{d}\leq\mathbf{0}^{\prime}$, there exists a theory $U\in {\sf D_I}$ with Turing degree $\mathbf{d}$?
  \item   Are elements of $\langle {\sf D_I}, \lhd\rangle$ comparable?
\end{enumerate}
\end{question}

In this paper, we examine above questions and prove more facts about $\langle {\sf D_T}, \leq_{\sf T}\rangle$ and $\langle {\sf D_I}, \lhd\rangle$. The structure of this paper is as follows. In Section 2, we list definitions and facts we use in this paper. In Section 3, we examine and prove some facts about the structure $\langle {\sf D_T}, \leq_{\sf T}\rangle$. In Section 4, we answer all questions about the structure $\langle {\sf D_I}, \lhd\rangle$ in \cite{Cheng20} and prove more results about $\langle {\sf D_I}, \lhd\rangle$. 
It is proved in \cite{PV} that  there are no minimal essentially undecidable theories with respect to interpretation. In Section 5, we generalize this result and give some general characterizations which tell us under what conditions there are no minimal RE theories having some property with respect to interpretation. In Section 6, we give some concluding remarks. In the Appendix, we give a brief overview of interpretation degree structures of three classes of theories  in the literature: general RE theories, RE theories extending $\mathbf{PA}$ and finitely axiomatizable theories.

\section{Preliminaries}

In this paper, we always assume the \emph{arithmetization} of the base theory. Given a sentence $\phi$, let $\ulcorner \phi\urcorner$ denote the G\"{o}del number of $\phi$. Under arithmetization, we equate a set  of sentences  with the set of G\"{o}del numbers of sentences.

Robinson Arithmetic $\mathbf{Q}$ and the theory $\mathbf{R}$ are both introduced in \cite{undecidable} by Tarski,  Mostowski and Robinson as  base axiomatic theories for investigating incompleteness and undecidability.

\begin{definition}[Robinson Arithmetic $\mathbf{Q}$, \cite{undecidable}]~\label{def of Q}
Robinson Arithmetic $\mathbf{Q}$  is  defined in   the language $\{\mathbf{0}, \mathbf{S}, +, \times\}$ with the following axioms:
\begin{description}
  \item[$\mathbf{Q}_1$] $\forall x \forall y\, (\mathbf{S}x=\mathbf{S} y\rightarrow x=y)$;
  \item[$\mathbf{Q}_2$] $\forall x\, (\mathbf{S} x\neq \mathbf{0})$;
  \item[$\mathbf{Q}_3$] $\forall x\, (x\neq \mathbf{0}\rightarrow \exists y\, (x=\mathbf{S} y))$;
  \item[$\mathbf{Q}_4$]  $\forall x\forall y\, (x+ \mathbf{0}=x)$;
  \item[$\mathbf{Q}_5$] $\forall x\forall y\, (x+ \mathbf{S} y=\mathbf{S} (x+y))$;
  \item[$\mathbf{Q}_6$] $\forall x\, (x\times \mathbf{0}=\mathbf{0})$;
  \item[$\mathbf{Q}_7$] $\forall x\forall y\, (x\times \mathbf{S} y=x\times y +x)$.
\end{description}
\end{definition}

Now, we introduce the theory $\mathbf{R}$, which contains all key properties of arithmetic for the proof of  $\sf G1$. The theory $\mathbf{R}$ has the Turing degree $\mathbf{0}^{\prime}$.

\begin{definition}[The theory $\mathbf{R}$, \cite{undecidable}]~
Let $\mathbf{R}$ be the theory consisting of $\mathbf{Ax1}$-$\mathbf{Ax5}$ in   the language $\{\mathbf{0}, \mathbf{S}, +, \times, \leq\}$ where  $\overline{n}=\mathbf{S}^n \mathbf{0}$ for $n \in \omega$.
\begin{description}
  \item[Ax1] $\overline{m}+\overline{n}=\overline{m+n}$;
  \item[Ax2] $\overline{m}\times\overline{n}=\overline{m\times n}$;
  \item[Ax3] $\overline{m}\neq\overline{n}$ if $m\neq n$;
  \item[Ax4] $\forall x(x\leq \overline{n}\rightarrow x=\overline{0}\vee \cdots \vee x=\overline{n})$;
  \item[Ax5] $\forall x(x\leq \overline{n}\vee \overline{n}\leq x)$.
\end{description}
\end{definition}

We introduce some basic notions in recursion theory.
\begin{definition}[Basic recursion theory, \cite{Rogers87}]\label{}~
\begin{enumerate}[(1)]
  \item Let $\langle \phi_e: e\in\omega\rangle$ be the list of all Turing programs, and $\langle {\sf W}_e: e\in\omega\rangle$ be the list of all RE sets, where
${\sf W}_e=\{x: \exists y \, {\sf T_1}(e,x, y)\}$ and ${\sf T_1}(z,x, y)$ is the Kleene predicate (cf.\cite{Kleene}).\footnote{Note that $x\in {\sf W}_e$ if and only if for some $y$,  the $e$-th Turing program with input $x$ yields an output in precisely $y$ steps.}
  \item Given $A, B\subseteq\omega$, we say that $A$ is \emph{Turing incomparable} with $B$ if $A\nleq_{\sf T} B$ and $B\nleq_{\sf T} A$.
  \item Let $\mathcal{C}$ be some class of theories. We say $S$ is a \emph{minimal}  RE theory in $\mathcal{C}$ w.r.t.  Turing reducibility if there is no RE theory $V$ in $\mathcal{C}$ such  that $V<_{\sf T} S$.
\end{enumerate}
\end{definition}

Now, we introduce some notions about incompleteness and undecidability in the literature. 
\begin{definition}[Incompleteness and undecidability, \cite{undecidable}]~\label{}
\begin{enumerate}[(1)]
\item We say a theory $T$ is \emph{incomplete} if there exists a sentence  in the language of $T$ which is neither provable nor refutable in $T$; otherwise, $T$ is \emph{complete}.
    \item A theory $T$
is \emph{locally finitely satisfiable} if every finitely axiomatizable sub-theory of $T$ has
a finite model.
    \item We say a theory $T$ is \emph{undecidable} if $T$ is not recursive; otherwise, $T$ is \emph{decidable}.
 \item We say a theory $T$ is \emph{hereditarily undecidable}  if any sub-theory $S$ of $T$ with the same language as $T$ is undecidable.
\end{enumerate}
\end{definition}


Now, we introduce some notions about interpretation and explain the meaning of minimality  with respect to interpretation. 

\begin{definition}[Interpretations]~
\begin{itemize}
\item Given theories $S$ and $T$, we say that  $S$ \emph{tolerates}  $T$ if $T$ is interpretable in some consistent extension of $S$ with the same language of $S$ (or  equivalently, for some interpretation $\tau$, the theory $S + T^{\tau}$ is consistent).\footnote{Note that $T^{\tau}$ should encompass the translations of the identity axioms and of the functionality axioms.}
    \item   We say that a theory $S$ is \emph{weaker} than a theory $T$ w.r.t.~ interpretation if $S\lhdneq T$.
 \item Given theories $S$ and $T$, we say that $S$ is \emph{incomparable} with $T$ w.r.t.~ interpretation if $S$ is not interpretable in $T$ and $T$ is not interpretable in $S$.
\item  Let $\mathcal{C}$ be some class of theories. We say $S$ is a \emph{minimal} RE theory in $\mathcal{C}$ w.r.t.~ interpretation if there is no RE theory $T$ in $\mathcal{C}$ such  that $T\lhdneq  S$.
\end{itemize}
\end{definition}

The notion of interpretation provides us with a method to compare different theories in different languages. If $T$ is interpretable in $S$, then all sentences provable (refutable) in $T$ are mapped, by the interpretation function, to sentences provable (refutable) in $S$.
The equivalence
classes of theories, under the equivalence relation $\bowtie$, are
called the interpretation degrees.

The following two theorems are about important properties of the theory $\mathbf{R}$ which we will use in this paper. 

\begin{theorem}[Visser, Theorem 6, \cite{Visser 14}]\label{visser thm on R}
For any RE theory $T$ with finite signature,  $T$ is interpretable in the theory $\mathbf{R}$ iff $T$ is locally finitely satisfiable.\footnote{In fact, if $T$ is locally finitely satisfiable, then $T$ is interpretable in $\mathbf{R}$ via a one-piece one-dimensional parameter-free interpretation.}
\end{theorem}

\begin{theorem}[Cobham, \cite{Visser16}]~\label{Cobham thm}
Any consistent RE theory that tolerates the theory $\mathbf{R}$ is undecidable.
\end{theorem}

The following theorem provides us with a method for proving essential undecidability of a theory via interpretation.

\begin{theorem}[Theorem 7, Corollary 2, \cite{undecidable}]\label{interpretable theorem}~
Let $T_1$ and $T_2$ be two consistent theories with finite signature such that $T_1$ is interpretable in $T_2$. If $T_1$ is essentially undecidable, then $T_2$ is essentially undecidable.
\end{theorem}

Now, we introduce the notion of recursively inseparable theories and effectively inseparable  theories. 

\begin{definition}[Inseparability, \cite{Smullyan}]~\label{Inseparability}
Let $T$ be a consistent RE theory, and $(A, B)$ be a pair  of disjoint RE sets.
\begin{enumerate}[(1)]
  \item The pair $(T_{\sf P}, T_{\sf R})$ is called the \emph{nuclei} of the theory $T$, where  $T_{\sf P}$ is the set of G\"{o}del numbers of sentences provable in $T$, and $T_{\sf R}$ is the set of G\"{o}del numbers of sentences refutable in $T$ (i.e., $T_{\sf P}=\{\ulcorner\phi\urcorner: T\vdash\phi\}$ and $T_{\sf R}=\{\ulcorner\phi\urcorner: T\vdash\neg\phi\}$).
      \item We say $(A, B)$ is a \emph{recursively inseparable} ($\sf RI$) pair  if there is no recursive set $X\subseteq\omega$ such that $A\subseteq X$ and $X\cap B=\emptyset$.
      \item We say $T$ is recursively inseparable $(\sf RI)$ if $(T_{\sf P}, T_{\sf R})$ is $\sf RI$.
    \item  We say $(A, B)$ is \emph{effectively inseparable} $(\sf EI)$ if there is a recursive function $f(x,y)$ such that for any $i$ and $j$, if $A\subseteq W_i$ and $B\subseteq W_j$ with $W_i\cap W_j=\emptyset$, then $f(i,j)\notin W_i\cup W_j$.
  \item We say $T$ is effectively inseparable $(\sf EI)$ if $(T_{\sf P}, T_{\sf R})$ is $\sf EI$.
\end{enumerate}
\end{definition}

We list some properties of recursively inseparable theories and effectively inseparable  theories we will use. 

\begin{fact}~\label{EI fact}
Let $S$  and $T$ be consistent RE theories.
\begin{enumerate}[(1)]
  \item If $S\lhd T$ and $S$ is $\sf RI \, ({\sf EI})$, then $T$ is $\sf RI \,({\sf EI})$ (\cite{Smullyan}).
      \item The theory $\mathbf{R}$ is ${\sf EI}$. As a corollary, if $\mathbf{R}\lhd T$, then $T$ is ${\sf EI}$ (\cite{Smullyan}).
  \item If $T$ is ${\sf EI}$, then $T$  has Turing degree $\mathbf{0}^{\prime}$ (\cite{Rogers87}).
      \item If $T$ is ${\sf EI}$, then $T$ is ${\sf RI}$, and if $T$ is ${\sf RI}$, then $T$ is ${\sf EU}$ (\cite{Smullyan}). 
\end{enumerate}
\end{fact}

\begin{definition}[\cite{PV}, \cite{Cheng20}]\label{sum of theory}
We  introduce two natural operators on RE theories. The infimum $U \oplus V$ is defined as follows: $U \oplus V$ is a theory in the
disjoint sum of the signatures of $U$ and $V$ plus a fresh $0$-ary
predicate symbol $P$. The theory is axiomatised by all $P \rightarrow\varphi$,
where $\varphi$ is a  sentence in $U$ plus $\neg P\rightarrow\psi$, where $\psi$ is a  sentence in $V$.
\end{definition}
We can show that $U \oplus V$  is the infimum
of $U$ and $V$ in the interpretability ordering $\lhd$. This result works for all choices of our
notion of interpretation (see \cite{Visser24}).

\section{The structure $\langle {\sf D_T}, \leq_{\sf T}\rangle$}

In this section, we examine the structure $\langle {\sf D_T}, \leq_{\sf T}\rangle$.
Hanf shows that there is a finitely axiomatizable theory in each recursively enumerable tt-degree (see \cite{Rogers87}).
Feferman shows in \cite{Feferman57} that if $A$ is any recursively enumerable set, then there is a recursively axiomatizable theory $T$ having the same Turing degree  as $A$.
In \cite{Shoenfield 58}, Shoenfield improves Feferman's result and shows that if $A$ is a non-recursive RE set, then there is an essentially undecidable theory having the same Turing degree as $A$.

\begin{theorem}[Shoenfield, \cite{Shoenfield 58}]\label{Shoenfield-1}
Suppose $A$ is a non-recursive RE set. Then: 
\begin{enumerate}[(1)]
  \item there is a recursively inseparable pair $\langle B, C\rangle$ such that $A$, $B$ and $C$ have the same Turing degree; 
  \item there is a consistent axiomatizable theory $T$ having one non-logical symbol which is essentially undecidable and has the same Turing degree as $A$.
\end{enumerate}
\end{theorem}

Thus, the structure $\langle {\sf D_T}, \leq_{\sf T}\rangle$ is as complex as the Turing degree structure of RE sets.

Janiczak's theory {\sf J} is introduced in \cite[p.136]{Janiczak}, and is used in Shoenfield's proof of Theorem \ref{Shoenfield-1} in \cite{Shoenfield 58}.
The theory {\sf J}  has only  one binary relation symbol $E$ with the
following axioms.
\begin{description}
  \item[{\sf J1}] $E$ is an equivalence relation.
  \item[{\sf J2}] There is at most one equivalence class of size precisely $n$.
  \item[{\sf J3}] There are at least $n$ equivalence classes with at least $n$ elements.\footnote{Our presentation of the theory {\sf J} follows \cite{PV}. We include the axiom {\sf J3} to make the proof of the following fact in Theorem \ref{J thm} more easy: over {\sf J}, every sentence is equivalent with a boolean combination of the $\Phi_n$'s.}
\end{description}
Let $\Phi_n$ denote the sentence: there exists an equivalence class of size precisely
$n + 1$. Note that the $\Phi_n$'s are mutually independent over {\sf J}.

\begin{theorem}[Janiczak, \cite{Janiczak}]~\label{J thm}
\begin{itemize}
  \item {\sf J} is decidable (see \cite[Theorem 4]{Janiczak}).
  \item Over {\sf J}, every sentence is equivalent with a Boolean combination of the $\Phi_n$'s (see \cite[Lemma 2]{Janiczak}), and  this Boolean combination  can be found explicitly from the given sentence.\footnote{This is a reformulation of Janiczak's Lemma 2 in \cite{Janiczak} in the context of {\sf J}. Janiczak's Lemma is proved by means of a method known as the elimination of quantifiers.}.
\end{itemize}
\end{theorem}

\begin{theorem}[]\label{general S thm}
Suppose $A$ is a non-recursive RE set. Then there is a partial recursive function $g$ such that the following holds, where $B_n=\{x: g(x)=n\}$: 
\begin{enumerate}[(1)]
  \item $(B_n, B_m)$ is a recursively inseparable pair for $m\neq n$ and each $B_n$ has the same Turing degree as $A$; 
  \item The theory $T_{(B_n, B_m)}={\sf J}+ \{\Phi_i: i\in B_n\} + \{\neg\Phi_i: i\in B_m\}$ for any $m\neq n$ is recursively inseparable and has the same Turing degree as $A$.
\end{enumerate}
\end{theorem}
\begin{proof}\label{}
Our proof is inspired by Exercise 10-20(ii)  in p.178 of \cite{Rogers87}.
Suppose $A$ is a non-recursive RE set. 
Let $f$ be the recursive function that enumerates $A$ without repetitions. Define a function $g$ as follows:
$g(\langle x,y\rangle)=n$ iff for some $s$, $f(s)=x$ and the program $\phi_y$ with input $\langle x,y\rangle$ yields $n$ as output in $< s$ steps. 

Note that $g$ is partial recursive. Define $B_n=\{x: g(x)=n\}$. Note that for any $n\neq m$, $B_n$ and $B_m$  are disjoint RE sets.

We first show that for any $m\neq n$, $(B_n, B_m)$ is a recursively inseparable pair and $B_n\equiv_{\sf T} B_m \equiv_{\sf T} A$.

\begin{claim}
$B_n\leq_{\sf T} A$.
\end{claim}
\begin{proof}\label{}
We test $\langle x,y\rangle\in B_n$ as follows. If $x\notin A$, then $\langle x,y\rangle\notin B_n$. Suppose $x\in A$. Take the unique $s$ such that $f(s)=x$. We can decide whether the program $\phi_y$ with input $\langle x,y\rangle$ yields $n$ in $< s$ steps. If yes, then $\langle x,y\rangle\in B_n$; if no, then $\langle x,y\rangle\notin B_n$.
\end{proof}

\begin{claim}
$A\leq_{\sf T} B_n$.
\end{claim}
\begin{proof}\label{}
We test $x\in A$ as follows. Suppose the index of the program with constant output value $n$ is $e_0$, i.e. $\phi_{e_0}(x)=n$ for all $x$. If $\langle x, e_0\rangle\in B_n$, then $x\in A$. Suppose $\langle x, e_0\rangle\notin B_n$, let $w$ be the number of computation steps such that $\phi_{e_0}(\langle x, e_0\rangle)=n$. Decide whether $f(s)=x$ and $s\leq w$ for some $s$. If yes, then $x\in A$; if no, then $x\notin A$.
\end{proof}

\begin{claim}
There is no recursive set $X$ that separates $B_n$ and $B_m$ (i.e., $B_n\subseteq X$ and $X\cap B_m=\emptyset$).
\end{claim}
\begin{proof}\label{}
Suppose not, i.e., we can find a recursive function $h$ such that ${\sf ran}(h)=\{m,n\}$, $h[B_n]=\{m\}$ and $h[B_m]=\{n\}$. Let $h=\phi_{e_1}$.
\begin{subclaim}
If $x\in A$, then the program $\phi_{e_1}$ with input $\langle x, e_1\rangle$ yields output in $\geq s$ steps where $f(s)=x$.
\end{subclaim}
\begin{proof}\label{}
Suppose not, i.e., $x\in A$, but the program $\phi_{e_1}$ with input $\langle x, e_1\rangle$ halts in less than $s$ steps where $f(s)=x$. Then by definition, we have $\phi_{e_1}(\langle x, e_1\rangle)=n \Leftrightarrow \phi_{e_1}(\langle x, e_1\rangle)=m$, which leads to a contradiction.
\end{proof}
Let $t$ be the recursive function such that $t(x)=$ the number of steps to compute the value of $\langle x, e_1\rangle$ in the program $\phi_{e_1}$. Then we have:
\[x\in A\Leftrightarrow\exists s (s\leq t(x)\wedge f(s)=x).\]
 Thus, $A$ is recursive which leads to a contradiction.
\end{proof}
Thus, for any  $n\neq m$, $(B_n, B_m)$ is a {\sf RI} pair and $B_n\equiv_{\sf T} B_m \equiv_{\sf T} A$.

Define the theory $T_{(B_n, B_m)}={\sf J}+ \{\Phi_i: i\in B_n\} + \{\neg\Phi_i: i\in B_m\}$ for any $m\neq n$. Now we show that $T_{(B_n, B_m)}$ has the same Turing degree as $A$.

Note that $T_{(B_n, B_m)}$ is a consistent RE theory. Since the $\Phi_n$'s are mutually independent over {\sf J}, we have $\Phi_i$ is
provable iff $i\in B_n$, and $\neg\Phi_i$ is provable iff $i\in B_m$. Thus, 
$B_n$ and $B_m$ are recursive in $T_{(B_n, B_m)}$. 
By Theorem \ref{J thm}, $T_{(B_n, B_m)}$ is
recursive in $B_n$ and $B_m$. 
Since $B_n\equiv_{\sf T} B_m \equiv_{\sf T} A$, $T_{(B_n, B_m)}\equiv_{\sf T} A$. Since $(B_n, B_m)$ is a {\sf RI} pair, by a standard argument, $T_{(B_n, B_m)}$ is recursively inseparable. 
\end{proof}

Theorem \ref{general S thm} improves Theorem \ref{Shoenfield-1} and shows that Shoenfield's theory $T$ in Theorem \ref{Shoenfield-1} is not unique. Since there are only countably many RE theories, Theorem \ref{general S thm} is the best result we can have. 

We list some  key results about the Turing degree structure of RE sets in the literature.

\begin{fact}[\cite{Rogers87}]~\label{key fact in recursion theory}
\begin{enumerate}[(1)]
  \item The RE degrees are dense: for any RE sets $A<_{\sf T} B$, there is a RE set $C$ such that $A <_{\sf T} C <_{\sf T} B$ (Sacks, Theorem 4.1, \cite{Soare}).
  \item There are no minimal non-zero RE degrees (follows from (1)).
  \item For any RE set  $\mathbf{0}<_{\sf T} A <_{\sf T} \mathbf{0}^{\prime}$, there exists an RE set $B$ such that $B$ is incomparable with $A$ w.r.t. Turing reducibility. Moreover, an index for $B$ can be found effectively from the index for $A$ (Yates, Theorem 4.3, \cite{Soare}).
  \item Given RE sets $A<_{\sf T} B$, there is an infinite RE sequence of RE sets $C_n$ such that $A<_{\sf T} C_n <_{\sf T} B$ and $C_n$'s are incomparable w.r.t. Turing reducibility (Robinson, p. 147, \cite{Soare}).
  \item If $a$ and $b$ are RE degrees such that $a<b$, then any countably partially ordered set can be embedded in the RE degrees between $a$ and $b$ (Robinson, p. 147, \cite{Soare}).
\end{enumerate}
\end{fact}

\begin{theorem}\label{}
\begin{enumerate}[(1)]
  \item $\langle {\sf D_T}, \leq_{\sf T}\rangle$ is dense: for any $A, B \in {\sf D_T}$ such that $A<_{\sf T} B$, there exist countably many theories $C\in {\sf D_T}$ such that  $A <_{\sf T} C <_{\sf T} B$.
   \item   $\langle {\sf D_T}, \leq_{\sf T}\rangle$ has no minimal elements.
  \item  For any theory $A\in {\sf D_T}$, there exist  countably many theories $B\in {\sf D_T}$ such that $A<_{\sf T} B$. Thus, $\langle {\sf D_T}, \leq_{\sf T}\rangle$ has  no maximal element.
  \item For any theory $A\in {\sf D_T}$, there exist countably many theories $B\in {\sf D_T}$ such that $B$ is incomparable with $A$ w.r.t. Turing reducibility.
  \item Given theories $A, B\in {\sf D_T}$ such that $A<_{\sf T} B$, there is an infinite RE sequence of theories $C_n\in {\sf D_T}$ such that $A<_{\sf T} C_n <_{\sf T} B$ and  $C_n$'s are incomparable w.r.t. Turing reducibility.
       \item Given  theories $A, B\in {\sf D_T}$ such that $A<_{\sf T} B$,  any countably partially ordered set can be embedded in  the structure $\langle {\sf D_T}, \leq_{\sf T}\rangle$ between $A$ and $B$.
           \item $\langle {\sf D_T}, \leq_{\sf T}\rangle$ is a dense distributive lattice without endpoints.
\end{enumerate}
\end{theorem}
\begin{proof}\label{}
(1): follows from Fact \ref{key fact in recursion theory}(1) and Theorem \ref{general S thm}. 

(2): follows from Fact \ref{key fact in recursion theory}(2).

(3): follows from Fact \ref{key fact in recursion theory}(1) and Theorem \ref{general S thm}. 

(4): follows from Fact \ref{key fact in recursion theory}(3) and Theorem \ref{general S thm}. 

(5): follows from Fact \ref{key fact in recursion theory}(4). 

(6): follows from Fact \ref{key fact in recursion theory}(5). 

(7): follows from (1)-(3). 
\end{proof}

\section{The structure $\langle {\sf D_I}, \lhd\rangle$}

The interpretation degree structure of essentially incomplete RE theories weaker than the theory $\mathbf{R}$ is much more complex. In this section, we answer all questions proposed in  \cite{Cheng20}, and prove more results about $\langle {\sf D_I}, \lhd\rangle$.

\begin{theorem}[\cite{Cheng20}]\label{main thm}
For any recursively inseparable pair $(A,B)$, there exists a consistent RE theory $T_{(A,B)}$ such that $\sf G1$  holds for $T_{(A,B)}$ and $T_{(A,B)}\lhdneq\mathbf{R}$.
\end{theorem}

Since by Theorem \ref{general S thm} there are countably many recursively inseparable pairs, by Theorem \ref{main thm}, the cardinality of ${\sf D_I}$ is $\aleph_0$.

Theorem \ref{Shoenfield first} is a reformulation of the proof of Theorem 1 in \cite{Shoenfield 58}. For proof details, we refer to \cite{PV}. 

\begin{theorem}[Shoenfield, \cite{Shoenfield 58}]\label{Shoenfield first}
Let  $A$ be a RE set. Then we can effectively find a disjoint pair $(B, C)$ of RE sets such that: 
\begin{enumerate}[(1)]
  \item $B, C\leq_{\sf T} \, A$; 
  \item for any RE set $D$ that separates $B$ and $C$ (i.e., $B\subseteq D$ and $D\cap C=\emptyset$), we have $A\leq_{\sf T} \, D$;
      \item  If $A$ is non-recursive, then $(B, C)$ is  recursively inseparable.
\end{enumerate}
\end{theorem}

There is no direct relation between the notion of interpretation and the notion of Turing reducibility. Given RE theories $U$ and $V$, $U\lhd V$ does not imply $U\leq_{\sf T} V$, and $U\leq_{\sf T} V$ does not imply $U\lhd V$. Now, we introduce the notion of Turing persistence which establishes the relationship between $U\lhd V$ and $U \leq_{\sf T} V$.

\begin{definition}[\cite{PV}]~\label{}
We say a RE theory $U$ is Turing persistent if for any consistent RE theory $V$, if $U\subseteq V$, then $U \leq_{\sf T} V$.
\end{definition}

Note that if $U$ is Turing persistent, then for any RE theory $V$, if $U\lhd V$, then $U \leq_{\sf T} V$. Note also that since any consistent RE theory interpreting $\mathbf{R}$ is ${\sf EI}$ and  has the Turing degree $\mathbf{0}^{\prime}$, hence any consistent RE theory interpreting $\mathbf{R}$ is Turing persistent.

Now we answer Question \ref{key qn}(1) in \cite{Cheng20}: can we show that for any Turing degree $\mathbf{0}<\mathbf{d}\leq\mathbf{0}^{\prime}$, there is a theory $U$ such that $\sf G1$  holds for $U$, $U\lhdneq \mathbf{R}$ and $U$ has Turing degree $\mathbf{d}$? Proposition \ref{main key thm} answers this question positively.

\begin{lemma}\label{hereditarily undecidable}
Let $T$ be a consistent RE theory. If $T$ tolerates $\mathbf{R}$, then $T$ is hereditarily undecidable.
\end{lemma}
\begin{proof}\label{}
Suppose $T$ tolerates $\mathbf{R}$, but $T$ is not hereditarily undecidable. Then there exists a sub-theory $S$ of $T$ with the same language as $T$ such that $S$ is decidable. Since $T$ tolerates $\mathbf{R}$, $S$ tolerates $\mathbf{R}$. By Theorem \ref{Cobham thm}, $S$ is undecidable, which leads to a contradiction.
\end{proof}

Proposition \ref{main key thm}(1) is due to Shoenfield (see Theorem 4.9 in \cite{PV}). For completeness, we give full details of the properties of the theory $T_{\mathbf{d}}$ in Proposition \ref{main key thm}. 

\begin{proposition}\label{main key thm}
Let $A$ be a RE set with Turing degree $\mathbf{d}$. Then we can effectively find a consistent RE theory $T_{\mathbf{d}}$ having one non-logical symbol such that 
\begin{enumerate}[(1)]
  \item $T_{\mathbf{d}}$ has the same Turing degree as $A$, and $T_{\mathbf{d}}$ is Turing persistent (see Theorem 4.9 in \cite{PV});
  \item $T_{\mathbf{d}}\lhdneq\mathbf{R}$;
\item  If $A$ is not recursive, then $T_{\mathbf{d}}$ is $\sf RI$. 
\end{enumerate}
\end{proposition}
\begin{proof}\label{}
Suppose $A$ is a RE set with Turing degree $\mathbf{d}$, $(B, C)$ is the disjoint pair  of RE sets effectively constructed from $A$ as in Theorem \ref{Shoenfield first}. 

(1) We show that $T_{\mathbf{d}}$ has the same Turing degree as $A$.
Define the theory $T_{(B,C)}={\sf J}+ \{\Phi_n: n\in B\} + \{\neg\Phi_n: n\in C\}$.
Let $T_{\mathbf{d}}$ denote the theory $T_{(B,C)}$. Note that $T_{\mathbf{d}}$ is a consistent RE theory. Since the $\Phi_n$'s are mutually independent over {\sf J}, we have   $\Phi_n$ is
provable iff $n\in B$, and $\neg\Phi_n$ is provable iff $n \in C$. Thus, 
$B$ and $C$ are recursive in $T_{\mathbf{d}}$.

By Theorem \ref{J thm}, $T_{\mathbf{d}}$ is
recursive in $B$ and $C$. 
Since $B, C\leq_{\sf T} A$ by Theorem \ref{Shoenfield first}, we have $T_{\mathbf{d}}$ is recursive in $A$. Since $B$ separates $B$ and $C$,  by Theorem \ref{Shoenfield first}, $A\leq_{\sf T} \, B$. Thus, since $B$ is recursive in $T_{\mathbf{d}}$, $T_{\mathbf{d}}$ has the same Turing degree as $A$.

Now we show that $T_{\mathbf{d}}$ is Turing persistent. 
Suppose $V$ is a consistent RE extension of $T_\mathbf{d}$. Define $D=\{n: V\vdash\Phi_n\}$. Note that $B\subseteq D$ and $D\cap C=\emptyset$. By Theorem \ref{Shoenfield first}, $A\leq_{\sf T} D$. Since $T_\mathbf{d}$ has the same Turing degree as $A$, we have $T_\mathbf{d}\equiv_{\sf T} A \leq_{\sf T} D \leq_{\sf T} V$. Thus, $T_\mathbf{d}$ is Turing persistent.

(2) We show that $T_{\mathbf{d}}\lhdneq \mathbf{R}$. By Theorem \ref{J thm}, every sentence is equivalent with a Boolean combination of the $\Phi_n$'s  over {\sf J}. Since the $\Phi_n$'s are mutually independent over {\sf J}, any finitely axiomatized sub-theory of $T_\mathbf{d}$
has a finite model. Thus, $T_\mathbf{d}$ is locally finitely satisfiable. Thus, by Theorem \ref{visser thm on R}, we have $T_\mathbf{d}\lhd\mathbf{R}$. We show that $\mathbf{R}$ is not interpretable in $T_\mathbf{d}$.
Note that the theory {\sf J}  is a decidable sub-theory of $T_\mathbf{d}$  in the same language. Since $T_\mathbf{d}$ is not hereditarily undecidable, by Lemma \ref{hereditarily undecidable}, $\mathbf{R}$ is not interpretable in $T_\mathbf{d}$. Thus $T_\mathbf{d}\lhdneq\mathbf{R}$.

(3) We show that if $A$ is non-recursive, then $T_{\mathbf{d}}$ is $\sf RI$. 
By Theorem \ref{Shoenfield first}, $(B,C)$ is ${\sf RI}$.  We  denote the theory $T_\mathbf{d}$ by $T$.
Define $g: n\mapsto \ulcorner \Phi_n\urcorner$. Note that $g$ is recursive, if $n\in B$, then $g(n)\in T_P$, and if $n\in C$, then $g(n)\in T_R$.
We show that $T$ is ${\sf RI}$. Suppose $T$ is not ${\sf RI}$, i.e., there is a recursive set $X$ such that $T_P\subseteq X$ and $X\cap T_R=\emptyset$. Note that $B\subseteq g^{-1}[T_P]\subseteq g^{-1}[X]$ and $C\subseteq g^{-1}[T_R]\subseteq g^{-1}[\overline{X}]=\overline{g^{-1}[X]}$. Since $g$ and $X$ are recursive, $g^{-1}[X]$ is recursive. Thus, $g^{-1}[X]$ is a recursive set separating $B$ and $C$, which  contradicts that $(B,C)$ is ${\sf RI}$. Thus, $T_{\mathbf{d}}$ is $\sf RI$.

Finally, we remark that the construction of the theory $T_\mathbf{d}$ is effective: there exists a recursive function $f$ such that if $A=W_e$, then there exists a RE theory $T$ with index $f(e)$ such that $T$ has the properties stated in Proposition \ref{main key thm}.
\end{proof}

\begin{corollary}
There is no minimal recursively inseparable (${\sf RI}$) RE theory w.r.t. Turing reducibility.
\end{corollary}

Note that for Turing persistent theories, if they are comparable w.r.t. interpretation, then they are comparable w.r.t. Turing reducibility. From Proposition \ref{main key thm}, given Turing incomparable non-recursive RE sets, we can find incomparable RE theories in ${\sf D_I}$ w.r.t. interpretation. Theorem \ref{incomparable ele} answers the following question in \cite{Cheng20} negatively: are elements of $\langle {\sf D_I}, \lhd\rangle$ comparable?

\begin{theorem}\label{incomparable ele}
Given RE sets $A<_{\sf T} B$, there is a sequence of   RE theories $\langle S_n: n\in\omega\rangle$ such that:
\begin{enumerate}[(1)]
  \item $S_n\in {\sf D_I}$;
  \item $A<_{\sf T} S_n <_{\sf T} B$;
  \item $S_n$ is Turing persistent;
  \item $S_n$ are incomparable w.r.t.~ interpretation.
\end{enumerate}
\end{theorem}
\begin{proof}\label{}
By Fact \ref{key fact in recursion theory}(4), there exists a sequence of RE sets $\langle C_n: n\in\omega\rangle$ such that $A<_{\sf T} C_n <_{\sf T} B$ and $C_n$ are incomparable w.r.t. Turing reducibility.
By Proposition \ref{main key thm}, for each $n$, we can find a Turing persistent RE theory $S_n\in {\sf D_I}$ with the same Turing degree as $C_n$. Thus, $A<_{\sf T} S_n <_{\sf T} B$ for each $n$. Since each $S_n$ is Turing persistent, and $C_n$'s are incomparable w.r.t. Turing reducibility, we have $S_n$'s are incomparable w.r.t. interpretation.
\end{proof}

As a corollary of Theorem \ref{incomparable ele}, there are countably many  incomparable elements of $\langle {\sf D_I}, \lhd\rangle$. This answers Question \ref{key qn}(2).

\begin{proposition}\label{no finite}
If $T\in {\sf D_I}$, then $T$ is not finitely axiomatizable.
\end{proposition}
\begin{proof}\label{}
Suppose $T\in {\sf D_I}$, but $T$ is finitely axiomatizable. Since $T\lhd \mathbf{R}$, $T$ is locally finitely satisfiable. Since $T$ is finitely axiomatizable, $T$ has a finite model, which contradicts the fact that $T$ is essentially undecidable.
\end{proof}

Proposition \ref{no finite} shows that there is no finitely axiomatizable theory interpretable in $\mathbf{R}$ for which $\sf G1$ holds.

\begin{lemma}[Lemma 4.8 in \cite{Cheng20}; Theorem 2.2 in \cite{PV}]~\label{key lemma}
For RE theories $A$ and $B$, if $\sf G1$ holds for both $A$ and $B$, then $\sf G1$ holds for $A\oplus B$.
\end{lemma}

\begin{definition}[\cite{Visser16}]\label{sum of theory}
Given consistent RE theories $U$ and $V$, the supremum theory $U \otimes V$ is defined as follows. 
The signature of $U \otimes V$ is the disjoint union of the signatures of $U$ and
$V$, plus two new unary predicates $\triangle_0$ and $\triangle_1$. The axioms of $U \otimes V$ are: 
\begin{enumerate}[(1)]
  \item  $P(x_0, \cdots, x_{n-1})\rightarrow \bigwedge_{i<n} \triangle_0(x_i)$, if $P$ is derived from the signature of $U$;
  \item $Q(y_0, \cdots, y_{m-1})\rightarrow \bigwedge_{j<m} \triangle_1(x_j)$, if $Q$ is derived from the signature of $V$;
   \item the axioms of $U$ relativized to $\triangle_0$;
  \item the axioms of $V$ relativized to $\triangle_1$;
  \item $\forall x  \,(\triangle_0(x)\vee \triangle_1(x))$;
  \item $\forall x \, \neg (\triangle_0(x)\wedge \triangle_1(x))$.
\end{enumerate}
\end{definition}

We know that the interpretation degree structures  of RE theories and essentially undecidable RE theories with the operators  $\oplus$ and $\otimes$ are a  distributive lattice.

\begin{theorem}\label{no maximal}
Let $A$ be a RE theory  with Turing degree less than $\mathbf{0}^{\prime}$
for which $\sf G1$ holds. Then we can effectively find a Turing persistent theory $S$, which is incomparable with $A$ with respect to Turing degrees, such that $\sf G1$ holds  for $S$ and $A\otimes S \ntriangleleft A$.
\end{theorem}
\begin{proof}\label{}
By Fact \ref{key fact in recursion theory}(3), we can effectively find a RE set $C$ such that $A$ is incomparable with $C$ w.r.t. Turing reducibility. By Proposition \ref{main key thm},  from $C$ we can  effectively find a  Turing persistent theory $S$ with the same Turing degree as $C$ such that  $\sf G1$ holds  for $S$. Let $B=A\otimes S$. Suppose  $B\lhd A$. Since $S\lhd B\lhd A$ and $S$ is Turing persistent,  we have $C\equiv_{\sf T} S\leq_{\sf T} A$, which leads to a contradiction. Thus, $A\lhdneq B$.
\end{proof}

\begin{fact}[Theorem XVI, \cite{Rogers87}]\label{R index}
The set $\{e: {\sf W}_e$ is recursive\} is $\Sigma^0_3$-complete.
\end{fact}

\begin{theorem}\label{V complete}
$\{e: \sf G1$ holds for ${\sf W}_e$\} is $\Pi^0_3$-complete.
\end{theorem}
\begin{proof}\label{}
Define $U=\{e: {\sf W}_e$ is undecidable\} and $V=\{e: \sf G1$ holds for ${\sf W}_e$\}.\footnote{The first  ${\sf W}_e$ in the definition of $U$ stands for a set of numbers, and the  second  ${\sf W}_e$ in the definition of $V$ stands for a theory as a set of theorems.}  By Fact \ref{R index}, $U$ is $\Pi^0_3$-complete. It is easy to check that $V$ is $\Pi^0_3$.

\begin{claim}\label{G1 index}
There is a recursive function $s$ such that for any $n\in \omega$:
\begin{enumerate}[(1)]
  \item  $n\in U\Leftrightarrow s(n)\in V$;
  \item ${\sf W}_n \equiv_{\sf T} {\sf W}_{s(n)}$.
\end{enumerate}
\end{claim}
\begin{proof}\label{}
Recall that the construction of the pair $(B,C)$ from a given RE set $A$  as in Theorem \ref{Shoenfield first} is effective: there are recursive functions $f$ and $g$ such that if $A={\sf W}_e$, then $B={\sf W}_{f(e)}$ and $C={\sf W}_{g(e)}$ such that if $A$ is non-recursive, then $(B,C)$ is recursively inseparable.
Recall that the construction of $T_{(B,C)}$ from a given disjoint pair $(B,C)$ of RE sets  in Proposition \ref{main key thm} is also effective: there is a recursive function $h(x)$ such that if $B={\sf W}_n$ and $C={\sf W}_m$ and $B\cap C=\emptyset$, then $T_{(B,C)}={\sf W}_{h(n,m)}$ such that $T_{(B,C)}$ has the same Turing degree as $A$.

Define $s(n)=h(f(n),g(n))$.
Note that $s$ is recursive.
Suppose $A={\sf W_e}$ and $(B,C)=({\sf W}_{f(e)}, {\sf W}_{g(e)})$  are constructed from $A$ as in Theorem \ref{Shoenfield first}.  Then $s(e)$ is the index of the theory $T_{(B,C)}$ constructed from $(B,C)$. By Proposition \ref{main key thm}, we have ${\sf W}_e \equiv_{\sf T} {\sf W}_{s(e)}$.

Now, we show that $e\in U$ iff $s(e)\in V$.
Suppose $e\in U$. Since ${\sf W}_e$ is non-recursive, $({\sf W}_{f(e)}, {\sf W}_{g(e)})$ is ${\sf RI}$ by Theorem \ref{Shoenfield first}. By Proposition \ref{main key thm}, $T_{({\sf W}_{f(e)}, {\sf W}_{g(e)})}={\sf W}_{s(e)}$ is recursively inseparable and hence essentially undecidable. Thus, $s(e)\in V$.
Suppose $e\notin U$. Since ${\sf W}_e$ is recursive and $T_{({\sf W}_{f(e)}, {\sf W}_{g(e)})}={\sf W}_{s(e)}\equiv_{\sf T} {\sf W}_e$, we have $s(e)\notin V$.
Thus $e\in U\Leftrightarrow s(e)\in V$.
\end{proof}

Thus, $V$ is $\Pi^0_3$-complete.
\end{proof}

\begin{remark}\label{important coro}
Let $U=\{e: {\sf W}_e$ is undecidable\} and $V=\{e: {\sf W}_e\in {\sf D_I}$ and ${\sf W}_e$ is Turing persistent\}.  Let $s$ be the recursive function defined in Theorem \ref{G1 index}.
Note that $s(e)$ is the index of the theory $T_{(B,C)}$  constructed in Proposition \ref{main key thm} where $(B,C)$ is constructed from $A={\sf W}_e$ as in Theorem \ref{Shoenfield first}.
If $e\in U$, by Proposition \ref{main key thm}, $s(e)\in V$. If $e\notin U$, then ${\sf W}_{s(e)}$ is decidable. Thus, $e\in U\Leftrightarrow s(e)\in V$.
\end{remark}

\begin{fact}[\cite{Rogers87}]~\label{more sets}
\begin{enumerate}[(1)]
  \item $\{e: {\sf W}_e$ is simple\} is $\Pi^0_3$-complete (Exercise 14-31(i), p. 334, \cite{Rogers87}).
  \item $\{e: {\sf W}_e$ is co-infinite\} is $\Pi^0_3$-complete (Corollary XVI, p. 328, \cite{Rogers87}).
\end{enumerate}
\end{fact}
As a corollary of Fact \ref{more sets} and Theorem \ref{V complete}, there exists a recursive function $f$ such that ${\sf W}_e$ is simple iff ${\sf G1}$ holds for ${\sf W}_{f(e)}$; and there exists a recursive function $g$ such that ${\sf W}_e$ is co-infinite iff ${\sf G1}$ holds for ${\sf W}_{g(e)}$.

\begin{theorem}[Murwanashyaka,  Pakhomov and Visser, Theorem 1.1 in \cite{PV}]\label{PV thm}
There are no minimal essentially undecidable recursively enumerable theories w.r.t. interpretation.
\end{theorem}

As a corollary of Theorem \ref{PV thm}, for any $T\in {\sf D_I}$, there exists $S\in {\sf D_I}$ such that $S\lhdneq T$. Thus, there are no minimal elements of $\langle {\sf D_I}, \lhd\rangle$. 
This answers an question in \cite{Cheng20}.

\section{Some generalizations}

The paper \cite{PV} gives two proofs of Theorem \ref{PV thm}. The first proof employs a direct diagonalisation argument and the second proof uses some recursion theoretic result.
In this section, we generalize the two proofs of Theorem \ref{PV thm} in \cite{PV} as in Theorem \ref{general first thm} and Theorem \ref{general important  thm}.

Now we first generalize the first proof of Theorem \ref{PV thm} in \cite{PV}.
Recall the Janiczak theory ${\sf J}$  and that the sentence $\Phi_n$ says that there exists an equivalence class of size precisely
$n + 1$.

\begin{definition}[\cite{PV}]\label{}
Given $X\subseteq \mathbb{N}$, we say that $W$ is a ${\sf J},X$-theory when $W$ is axiomatised over {\sf J} by boolean combinations of the sentences $\Phi_n$ for $n \in X$.
\end{definition}

\begin{theorem}[Theorem 4.5 in \cite{PV}]\label{key thm of first proof}
Given any essentially undecidable  theory $U$, we can effectively find a recursive set $X$ (from an index of $U$) such that no consistent ${\sf J},X$-theory interprets $U$. As a corollary, for any essentially undecidable theory $U$, we can effectively find an essentially undecidable theory $W$ such that $U$ is not interpretable in $W$.
\end{theorem}

From Theorem \ref{key thm of first proof}, for any essentially undecidable theory $U$, we can effectively find an essentially undecidable theory $V$ such that $V \lhdneq U$.\footnote{Take the essentially undecidable  theory $W$ from Theorem \ref{key thm of first proof} such that $U$ is not interpretable in $W$. Let $V=U\oplus W$. By Lemma \ref{key lemma}, $V$ is essentially undecidable. Since $U$ is not interpretable in $W$, we have $U$ is not interpretable in $V$. Thus, $V \lhdneq U$.}

\begin{lemma}\label{EI lemma}
Suppose $(Y,Z)$ is an effectively inseparable pair of RE sets.
Define the theory $V={\sf J} + \{\Phi_n: n \in Y\}+ \{\neg \Phi_n: n \in Z\}$.
Then $V$ is {\sf EI}.
\end{lemma}
\begin{proof}\label{}

We want to find a recursive function $h(i,j)$ such that if $V_P\subseteq {\sf W}_i, V_R\subseteq {\sf W}_j$ and ${\sf W}_i\cap {\sf W}_j=\emptyset$, then $h(i,j)\notin {\sf W}_i\cup {\sf W}_j$.

Define the function $f: n\mapsto \Phi_n$. Clearly, $f$ is recursive. By s-m-n theorem, there is a recursive function $g$ such that $f^{-1}[{\sf W}_i]={\sf W}_{g(i)}$. Suppose $(Y,Z)$ is effectively  inseparable via the recursive function $t(i,j)$. Define $h(i,j)=f(t(g(i), g(j)))$. Clearly, $h$ is recursive.

Suppose
$V_P\subseteq {\sf W}_i, V_R\subseteq {\sf W}_j$ and ${\sf W}_i\cap {\sf W}_j=\emptyset$.
Note that $Y\subseteq f^{-1}[V_P]\subseteq f^{-1}[{\sf W}_i]={\sf W}_{g(i)}$, and $Z\subseteq f^{-1}[V_R]\subseteq f^{-1}[{\sf W}_j]={\sf W}_{g(j)}$.
Note that ${\sf W}_{g(i)}\cap {\sf W}_{g(j)}=\emptyset$ since ${\sf W}_i\cap {\sf W}_j=\emptyset$.
Since $(Y,Z)$ is {\sf EI} via the function $t$, we have $t(g(i), g(j))\notin {\sf W}_{g(i)}\cup {\sf W}_{g(j)}$. Then, $h(i,j)\notin {\sf W}_i\cup {\sf W}_j$. Thus, $V$ is {\sf EI}.
\end{proof}

\begin{theorem}\label{general first thm}
Let $P$ be some property of RE theories, and $S, T$ be any consistent RE theories. Suppose the following conditions hold:
\begin{enumerate}[(1)]
  \item the operator $\oplus$ is closed under the property $P$: if $S$ and $T$ have the property $P$, then $S\oplus T$ also has the property $P$;
  \item if $S$ has the property $P$, then $S$ is essentially undecidable;
  \item if $S$ is {\sf EI}, then $S$ has the property $P$.
\end{enumerate}
Then for any RE theory $U$ with the property $P$, we can effectively find a  RE theory $T$ with the property $P$ such that $T\lhdneq U$. As a corollary, there are no minimal RE theories with the property $P$ w.r.t. interpretation. 
\end{theorem}
\begin{proof}\label{}
Let $U$ be a consistent RE theory with the property $P$. By Condition (2), Theorem \ref{key thm of first proof} applies to RE theories with the property $P$. Thus, we can effectively find a recursive set $X$ (from an index of $U$) such that no consistent
${\sf J},X$-theory interprets $U$. Let $Y$ and $Z$ be effectively inseparable RE sets which are subsets of $X$.
Define the theory $V$ as $V={\sf J} + \{\Phi_n: n \in Y\}+ \{\neg \Phi_n: n \in Z\}$.
By Lemma \ref{EI lemma}, $V$ is {\sf EI}. By Condition (3), $V$ has the property $P$. Let $T=U\oplus V$. By condition (1), $T$ has the property $P$. Since no consistent
${\sf J},X$-theory interprets $U$, $U$ is not interpretable in $V$.
Thus, $T\lhdneq U$ since $U$ is not interpretable in $T$.
\end{proof}

Now we generalize the second proof of Theorem \ref{PV thm} in \cite{PV}.
We say $\sqsubseteq$ is a \emph{partial pre-ordering} on RE theories if $\sqsubseteq$ is reflexive and transitive: i.e., for any RE theory $S, T$ and $U$, we have $S\sqsubseteq S$, and if $S\sqsubseteq T$ and $T\sqsubseteq U$, then $S\sqsubseteq U$. Given any partial pre-ordering  $\sqsubseteq$ on RE theories, we define $S\equiv T \triangleq S\sqsubseteq T$ and  $T\sqsubseteq S$. From this equivalence relation, we can introduce a degree structure generated from this equivalence relation. Let $T\sqsubset S$ denote that  $T\sqsubseteq S$ but $S\sqsubseteq T$ does not hold. We say  $S$ is a \emph{minimal RE theory w.r.t. the relation $\sqsubseteq$} if there is no RE theory $T$ such that $T\sqsubset S$.

Let $P$ be any property of RE theories (e.g., {\sf EU}, {\sf EI}, etc). Let ${\sf D}_P$ be the collection of RE theories with the property $P$.
Theorem \ref{general important  thm} provides us with a general way to show that $\langle {\sf D}_P, \sqsubseteq\rangle$ has no minimal elements w.r.t. the degree structure induced from the relation $\sqsubseteq$.

\begin{theorem}\label{general important  thm}
Given a property $P$ of RE theories, suppose $\sqsubseteq$ is a partial pre-ordering on RE theories satisfying the following conditions for some $n\in\omega$:
\begin{enumerate}[(1)]
  \item For any $A, B\in {\sf D}_P$, there is a  lower bound $A\boxplus B$ of $A$ and $B$ under $\sqsubseteq$ such that $A\boxplus B\in {\sf D}_P$;
  \item If $A\in {\sf D}_P$ and $A\sqsubseteq B$, then  $B\in {\sf D}_P$;
  \item The relation $A\sqsubseteq B$ is $\Sigma^0_n$;
  \item  $V=\{e: {\sf W}_e$ has the property $P$\} is $\Pi^0_n$-complete.\footnote{Here, ``${\sf W}_e$ has the property $P$" means that the RE theory with index $e$ has the property $P$.}
\end{enumerate}
Then $\langle {\sf D}_P, \sqsubseteq\rangle$ has no minimal element.
\end{theorem}
\begin{proof}\label{}
Suppose $\langle {\sf D}_P, \sqsubseteq\rangle$ has a minimal element $T$. We show that for any RE theory $A\in {\sf D}_P$, we have $T\sqsubseteq A$.
Note that $T\boxplus A\in {\sf D}_P$. Since $T$ is a minimal element of ${\sf D}_P$, and $T\boxplus A\in {\sf D}_P$ is a lower bound of $T$ and $A$ by condition (1), we have $T\sqsubseteq T\boxplus A$. Since $T\boxplus A\sqsubseteq A$, we have $T\sqsubseteq A$.

Note that by condition (2), $e \in V \Leftrightarrow T\sqsubseteq {\sf W}_e$.
By condition (3), $V$ is $\Sigma^0_n$ which contradicts condition (4).
\end{proof}

\begin{theorem}\label{general third thm}
Let $P$ be some property of RE theories, and $S, T$ be any consistent RE theories. Suppose the following conditions hold:
\begin{enumerate}[(1)]
  \item the operator $\oplus$ is closed under the property $P$: if $S$ and $T$ have the property $P$, then $S\oplus T$ also has the property $P$;
  \item For any essentially undecidable RE theory $U$, we can effectively
find (an index of) a RE theory $V$ with the property $P$ such that
$U\ntriangleleft V$.
\item If a theory has the property $P$, then it is essentially undecidable.
\end{enumerate}
Then for any RE theory $U$ with the property $P$, we can effectively find a  RE theory $T$ with the property $P$ such that $T\lhdneq U$. 
\end{theorem}
\begin{proof}\label{}
Let $U$ be a RE theory with the property $P$. By Condition (3), $U$ is essentially undecidable. By Condition (2), we can effectively
find (an index of) a RE theory $V$ with the property $P$ such that
$U\ntriangleleft V$. Let $T=U \oplus V$. By Condition (1), $T$ has the property $P$. Since $U\ntriangleleft V$, we have $T\lhdneq U$. 
\end{proof}

Note that Theorem \ref{general first thm}  and Theorem \ref{general third thm} tell us more information than Theorem \ref{general important  thm}. The proof of Theorem \ref{general important  thm} is not effective: it shows that  there are no minimal  theories with the property $P$ w.r.t. interpretation, but it does not tell us that given a theory with the property $P$, we can effectively find a weaker theory with the property $P$ w.r.t. interpretation. But Theorem \ref{general first thm} and Theorem \ref{general third thm} also tell us that we can effectively find a weaker theory with the property $P$ w.r.t. interpretation, given a theory with the property $P$. 

We can view the second proof of Theorem \ref{PV thm} in \cite{PV} as a special case of Theorem \ref{general important  thm}  in which $n=3, P$ is the property of essential undecidability, $\sqsubseteq$ is the interpretation relation, and $A\boxplus B$ is $A\oplus B$. 
We can view the first proof of Theorem \ref{PV thm} in \cite{PV} as a special case of Theorem \ref{general first thm} and Theorem \ref{general third thm} where $P$ is the property of essential undecidability.

\begin{theorem}[Theorem 3.15, \cite{Cheng EI}]~\label{}
There are no minimal effectively inseparable theories with respect to interpretability: for any effectively inseparable theory $T$, we can effectively find a theory which is effectively inseparable and strictly weaker than $T$ w.r.t. interpretation.
\end{theorem}
We can view the proof of Theorem 3.15 in \cite{Cheng EI} as a special case of Theorem \ref{general first thm} and Theorem \ref{general third thm} where $P$ is the property of effective inseparability.

Before we discuss the application of Theorem \ref{general first thm}, Theorem \ref{general important  thm} and Theorem \ref{general third thm} to essentially hereditarily undecidable theories, we first introduce the notion of essentially hereditarily undecidable theories. 
  
\begin{definition}[Essentially hereditarily undecidable theories, \cite{undecidable}]\label{def of 1.2}
Let $T$ be a consistent RE theory.
\begin{enumerate}[(1)]
     \item We say $T$ is \emph{hereditarily undecidable} ({\sf HU}) if every sub-theory $S$ of $T$ over the same language  is undecidable.
       \item We say $T$ is \emph{essentially hereditarily undecidable} ({\sf EHU}) if any consistent RE extension of $T$ is {\sf HU}.
\end{enumerate}
\end{definition}

Visser gave two proofs that there is no interpretability minimal essentially hereditarily undecidable
RE theory: Theorem 41  and Theorem 44 in \cite{Visser24}. 

\begin{theorem}[\cite{Visser24}]~\label{Visser ger}
\begin{enumerate}[(1)]
  \item For any essentially undecidable RE theory $U$, we can effectively
find (an index of) an essentially hereditarily undecidable RE theory $V$ such that
$U\ntriangleleft V$ (Theorem 44 in \cite{Visser24}). 
  \item There is no interpretability minimal essentially hereditarily undecidable
RE theory (Theorem 41 in \cite{Visser24}).
\end{enumerate}
\end{theorem}

Visser proved in \cite{Visser24} that essentially hereditarily undecidable theories have the following properties:
\begin{enumerate}[(1)]
  \item  If $S$ and $T$ are essentially hereditarily undecidable theories, then $S\oplus T$ is also  essentially hereditarily undecidable (Theorem 15, \cite{Visser24}); 
\item Suppose $U$ is consistent and essentially hereditarily undecidable and
$U\lhd V$. Then $V$ is essentially hereditarily undecidable (Theorem 16, \cite{Visser24}).
\end{enumerate}

It is easy to check that if $T$ is essentially hereditarily undecidable, then $T$ is essentially undecidable. From the argument of Theorem 41 in \cite{Visser24}, we can show that $\{e: {\sf W}_e$ is essentially hereditarily undecidable\} is $\Pi^0_3$-complete.
From these properties, it is easy to see that Theorem \ref{general important  thm}  and Theorem \ref{general third thm} apply to 
essentially hereditarily undecidable theories. 
We can view the proof of Theorem 44 in \cite{Visser24} as a special case of Theorem \ref{general third thm} where $P$ is the property of essential hereditary undecidability.  We can view the proof of Theorem 41 in \cite{Visser24} as a special case of Theorem \ref{general important  thm} in which $n=3, P$ is the property of essential hereditary undecidability, $\sqsubseteq$ is the interpretation relation, and $A\boxplus B$ is $A\oplus B$. However, Condition (3) of Theorem \ref{general first thm} does not apply to  essentially hereditarily undecidable theories since effective inseparability does not imply essential hereditary undecidability
(see Theorem 5.4 in \cite{Cheng 24}). 

We conclude this section with some comments about {\sf Creative} theories and Rosser theories. 
We first give some definitions. Let $T$ be a consistent RE theory in the language of arithmetic or in a language admitting numerals, and $(A,B)$ be a disjoint pair of RE sets.
We say $T$ is {\sf Creative} if $T_P$ is creative.
We say $(A,B)$ is \emph{separable} in $T$ if there is a formula $\phi(x)$ with only one free variable such that if $n\in A$, then $T\vdash \phi(\overline{n})$, and if $n\in B$, then $T\vdash \neg\phi(\overline{n})$. We say $T$ is  \emph{Rosser}  if any disjoint  pair of RE sets is separable in $T$.

Note that $\{e: {\sf W}_e$ is creative\} is $\Sigma^0_3$-complete (Exercise 14-31(iii), p. 334 in \cite{Rogers87}). Thus, Condition (4) in Theorem \ref{general important  thm} does not apply to {\sf Creative} theories.  Being {\sf Creative}  does not imply being essentially undecidable (see Theorem 4.12 in \cite{Cheng 24}).\footnote{For example, let  $X$ be a creative set.  Define $T:={\sf J}+ \{\Phi_n: n\in X\}$. It is easy to check that $T$ is {\sf Creative}, but ${\sf J}+ \{\Phi_n \mid n\in \omega\}$ is a consistent complete decidable RE extension of $T$ (Theorem 4.12 in \cite{Cheng 24}).} Thus, Condition (2) in Theorem \ref{general first thm}  and Condition (3) in Theorem \ref{general third thm} does not apply to {\sf Creative} theories. 

For Rosser theories, we can show that if $S$ and $T$ are Rosser theories, then $S\oplus T$ is also Rosser as follows. Let $(A,B)$ be a disjoint pair of RE sets. Since $S$ is Rosser, there exists a formula $\phi(x)$ such that $n\in A\Rightarrow S\vdash \phi(\overline{n})$ and $n\in B\Rightarrow S\vdash \neg\phi(\overline{n})$. Since $T$ is Rosser, there exists a formula $\psi(x)$ such that $n\in A\Rightarrow T\vdash \psi(\overline{n})$ and $n\in B\Rightarrow T\vdash \neg\psi(\overline{n})$. Define $\varphi(x):= (P\rightarrow \phi(x))\wedge (\neg P\rightarrow \psi(x))$. We can check that $n\in A\Rightarrow S\oplus T\vdash \varphi(\overline{n})$ and $n\in B\Rightarrow S\oplus T\vdash \neg\varphi(\overline{n})$.
However, we do not know whether for any essentially undecidable RE theory $U$, we can effectively
find (an index of) a Rosser theory $V$  such that
$U\ntriangleleft V$. We did not explore this in this work. 

\section{Conclusion}

We give some concluding remarks.
The existence of a non-recursive RE set is essential to understand the incompleteness phenomenon. Given any  non-recursive RE set $A$, we can effectively construct a RE theory with the same Turing degree as $A$ for which $\sf G1$ holds. In \cite{Friedman21}, Harvey Friedman proves $\sf G2$ for theories interpreting $\sf I\Sigma_1$ based on the existence of a remarkable set which is equivalent to the  existence of a non-recursive RE set.  Whether there are  minimal RE theories for which $\sf G1$ holds depends on the definition of minimality. The fact that there are no minimal RE theories  for which $\sf G1$ holds w.r.t. Turing reducibility and interpretation shows that the incompleteness phenomenon is omnipresent, and there is no limit of $\sf G1$ w.r.t. Turing reducibility and interpretation.

Both \cite{Cheng20} and this paper are about the limit of  $\sf G1$. A natural question is: what is the limit of the second incompleteness theorem ($\sf G2$) if any?
Both mathematically and philosophically, $\sf G2$ is  more problematic than $\sf G1$. In the case of $\sf G1$, we are mainly interested in the fact that \emph{some} sentence is independent of the base theory. But in the case of $\sf G2$, we are also interested in the content of the consistency statement.
We can say that ${\sf G1}$ is extensional in the sense that we can construct a concrete mathematical statement which is independent of the base theory without referring to  arithmetization or  provability predicate. However, ${\sf G2}$ is intensional and ``whether the consistency of a theory $T$ is provable in $T$" depends on many factors such as the way of formalization, the base theory  we use,  the way of coding, the way to express consistency, the provability predicate we use, etc. For the discussion of the intensionality of ${\sf G2}$, we refer to \cite{Cheng 19}.

Let $T$ be a consistent RE theory  in the language of arithmetic or in a language admitting numerals via interpretation. We say $\sf G2$ holds for $T$ if $T\nvdash \mathbf{Con}(T)$.\footnote{We assume that $\mathbf{Con}(T)$ is the canonical arithmetic formula expressing the consistency of the base theory $T$ defined as $\neg \mathbf{Pr}_T(\mathbf{0}\neq \mathbf{0})$ saying that $\mathbf{0}\neq \mathbf{0}$ is not provable in $T$ where the coding we use is the standard G\"{o}del coding and the provability predicate we use satisfies Hilbert-Bernays-L\"{o}b derivability conditions.} Pudl\'{a}k shows that no RE theory $T$ (in predicate logic) interprets $\mathbf{Q}+\mathbf{Con}(T)$ (see \cite{Pudlak 85}). As a corollary, $\sf G2$ holds for $\mathbf{Q}$. Authors in \cite{Bezboruah 76} proved that $\sf G2$ holds for $\mathbf{Q}$ using a totally different method. 

\begin{theorem}[Theorem 1, \cite{Bezboruah 76}]\label{theory for Q}
Let $L$ be any formal system with a recursive set of axioms, a finite number of finitary and recursive rules of inference including modus ponens and having $A\rightarrow A$ as a theorem for all sentences $A$. Then $Con(L)$ is not provable in $T_0$, where $T_0$ is a weak arithmetic theory with the property that $\mathbf{Q}\subseteq T_0\subseteq \mathbf{PA}$.\footnote{For the definition of $T_0$, we refer to \cite{Bezboruah 76}.} 
\end{theorem}

I list  the question whether $\sf G2$ holds for $\mathbf{R}$ (i.e.~ whether $\mathbf{R}\nvdash {\sf Con(\mathbf{R})}$ holds) as an open question in \cite{Cheng 19}. In fact, as a corollary of Theorem \ref{theory for Q}, since $\mathbf{R}\subseteq \sf PA$, we have $\sf G2$ holds for the theory $\mathbf{R}$.

We know that $\sf G1$ holds for many RE theories $T$ such that $T$ has Turing degree less than $\mathbf{0}^{\prime}$ and $T \lhdneq \mathbf{R}$. 
However, we do not know whether there exists a RE theory $T$ such that $\sf G2$ holds for $T$, $T$ has Turing degree less than $\mathbf{0}^{\prime}$ and $T \lhdneq \mathbf{R}$. We did not explore this question.

We conclude this paper with some informal comments. We know that given an essentially incomplete (effectively inseparable, essentially hereditarily undecidable) RE theory, we can effectively find another strictly weaker essentially incomplete (effectively inseparable, essentially hereditarily undecidable) theory w.r.t. interpretation. 
These facts reveal that the first incompleteness theorem is limitless, and provide us with more evidences from logic that the incompleteness phenomenon is ubiquitous (pervasive) in  abstract formal theories. From the viewpoint of logic, in matters of incompleteness everything is always as bad as it can get via the magic tools from logic, and these negative results are just weird facts about an already strange phenomenon. 
However, if we only consider those ``natural" mathematical theories, my personal view is that the theory $\mathbf{R}$ is the weakest ``natural" mathematical theory which is essentially incomplete even if we do not have a formal definition of naturalness in the literature.

\section{Appendix}

In this Appendix, we give a brief overview of interpretation degree structures of three classes of theories  in the literature: general RE theories, RE theories extending $\mathbf{PA}$ and finitely axiomatized theories. In this Appendix, we use the most inclusive notion of interpretation as explained in the introduction section.

\begin{definition}[Distributive lattice, \cite{Per 97}]~\label{}
\begin{enumerate}[(1)]
  \item A lattice is an algebraic structure $(L,\vee ,\wedge)$, consisting of a set $L$ and two binary, commutative and associative operations $\vee$  and $\wedge$ on $L$ satisfying the following axiomatic identities for all elements $a,b\in L$ (sometimes called absorption laws):
\begin{enumerate}[(A)]
  \item $a\vee (a\wedge b)=a$;
  \item $a\wedge (a\vee b)=a$.
\end{enumerate}
  \item
A lattice is distributive if it satisfies $a\vee (b\wedge c)=(a\vee b)\wedge (a\vee c)$ for all $a,b,c\in L$.
\end{enumerate}
\end{definition}

We first examine the interpretation degree structure of general RE theories.
The interpretation degree structure of RE theories extending $\mathbf{PA}$ is well known in the literature. We have some equivalent characterizations of the interpretation relation between consistent RE extensions of $\mathbf{PA}$. Given a theory $T$, define $T \upharpoonright k = \{n\in T: n \leq k\}$.

\begin{fact}[Orey-H\'{a}jek Characterization, \cite{Per 97}]
Suppose theories $S$ and $T$ are consistent RE extensions of $\mathbf{PA}$. Then the following are equivalent:
\begin{enumerate}[(1)]
  \item $S\lhd T$;
  \item $T\vdash \mathbf{Con}(S\upharpoonright k)$ for any $k\in\omega$;
  \item $(S\upharpoonright k)\lhd T$ for any $k\in\omega$;
  \item For any $\Pi^{0}_1$ sentence $\phi$, if $S\vdash \phi$, then $T\vdash \phi$.
  \item For any model $M$ of $T$, there is a model $N$ of $S$ such that $M$ is isomorphic to an initial segment of $N$.
\end{enumerate}
\end{fact}

Let $\langle {\sf D_\mathbf{PA}}, \lhd\rangle$ denote the interpretation degree structure of consistent RE extensions of $\mathbf{PA}$. 
In fact, $\langle {\sf D_\mathbf{PA}}, \lhd, \downarrow, \uparrow\rangle$ is a dense distributive lattice\footnote{We say that $\langle {\sf D_\mathbf{PA}}, \lhd, \downarrow, \uparrow\rangle$ is dense if for any $A,B\in {\sf D_\mathbf{PA}}$ such that $A\lhdneq B$, there exists $C\in {\sf D_\mathbf{PA}}$ such that $A\lhdneq C\lhdneq B$.}, where 
\[A\downarrow B=T+\{\mathbf{Con}(A\upharpoonright k)\vee \mathbf{Con}(B\upharpoonright k): k\in\omega\}\]
 and
\[A\uparrow B= T+\{\mathbf{Con}(A\upharpoonright k)\wedge \mathbf{Con}(B\upharpoonright k): k\in\omega\}.\]
For more properties of $\langle {\sf D_\mathbf{PA}}, \lhd\rangle$, we refer to \cite{Per 97}.

The interpretation degree structure of finitely axiomatized theories, denoted by $\langle {\sf D_{fin}}, \lhd\rangle$, is studied by
Harvey Friedman in \cite{Friedman07}. Note that there are only  $\aleph_0$ many interpretation degrees of finitely axiomatized theories. The structure $\langle {\sf D_{fin}}, \lhd\rangle$ forms a (reflexive)
partial ordering with a minimum element $\top$ and a maximum element $\bot$, where  $\top$ is the equivalence
class of all sentences with a finite model, and $\bot$ is the equivalence class of all sentences
with no models.

\begin{theorem}[\cite{Friedman07}]~
\begin{enumerate}[(1)]
  \item The structure $\langle {\sf D_{fin}}, \lhd, \oplus, \otimes\rangle$ is a distributive lattice.
   \item The structure $\langle {\sf D_{fin}}, \lhd\rangle$ is dense, i.e., $a \lhdneq b \rightarrow (\exists c)(a \lhdneq c
\lhdneq b)$.
  \item For any  $a, b\in {\sf D_{fin}}$, if $a \lhdneq b$, then there
exists an infinite sequence $c_n$  such that $a\lhdneq c_n\lhdneq b$ for each $n$ and $c_n$'s are incomparable w.r.t. interpretation.
\end{enumerate}
\end{theorem}

Thus, $\langle {\sf D_{fin}}, \lhd, \oplus, \otimes\rangle$  is a dense distributive lattice.
The structure $\langle {\sf D_{fin}}, \lhd\rangle$ bears a rough resemblance to the Turing degree structure of RE sets.

\end{document}